\documentclass{amsart}
\usepackage[utf8]{inputenc}

\usepackage[x11names]{xcolor}
\usepackage{ae,aecompl, stmaryrd}

\title[Lyapunov exponents in families]{An asymptotic for sums of Lyapunov exponents in families}
\author{Patrick Ingram}
\address{Department of Mathematics and Statistics, York University, Toronto, Canada}
\email{pingram@yorku.ca}
\author{David Jaramillo-Martinez} 
\address{Department of Mathematics and Statistics, York University, Toronto, Canada}
\email{djm738@yorku.ca}
\author{Jorge Mello}
\address{Department of Mathematics and Statistics, Oakland University, Rochester Hills, MI, USA}
\email{jorgedemellojr@oakland.edu}

\date{\today}

\newcommand{\PP}{\mathbb{P}}
\newcommand{\CC}{\mathbb{C}}
\newcommand{\QQ}{\mathbb{Q}}
\newtheorem{theorem}{Theorem}
\newtheorem{lemma}[theorem]{Lemma}
\theoremstyle{definition}
\newtheorem*{remark}{Remark}

\begin{document}
\nocite{*}
\begin{abstract}
Let $f_t$ be a meromorphic family of endomorphisms of $\PP^N_\CC$ of degree at least 2, and let $L(f_t)$ be the sum of Lyapunov exponents associated to $f_t$. Favre showed that
\[L(f_t)=L(f_\eta)\log|t^{-1}|+o(\log|t^{-1}|)\]
as $t\to 0$, where $L(f_\eta)$ is the sum of Lyapunov exponents on the generic fibre, interpreted as an endomorphism of some projective Berkovich space. Under some additional constraints on the family $f_t$, we provide an explicit error term.
\end{abstract}
\maketitle

\section{Introduction}

The sum of Lyapunov exponents of $f:\PP^N_\CC\to\PP^N_\CC$ quantifies the average rate of expansion, and is defined by the integral
\[L(f)=\int \log\|\det(Df)\|d\mu,\]
where $\mu$ is the measure of maximal entropy associated to $f$, and $\|\cdot\|$ is the norm in terms of the Fubini-Study metric.
From work of  Bassanelli and Bertoloot \cite{BB}, it is known that $L(f)$ varies nicely in holomorphic families, and so it is natural to ask what happens as a meromorphic family degenerates to a pole. As a sample application of a deeper result on degeneration of measures on hybrid spaces, Favre proved the following.

\begin{theorem}(Favre, \cite[Theorem C]{F})
Let $f_t$ be a meromorphic family of endomorphisms of $\PP^N$ parametrized by the unit disk, and let $f_\eta$ be the endomorphism induced by this family by the Berkovich space associated to $\CC(\!( t )\!)$. Then
\[L(f_t)=L(f_\eta)\log|t^{-1}|+o(\log|t^{-1}|),\]
as $t\to 0$, where $o(x)/x\to 0$ as $x\to 0$.
\end{theorem}
The purpose of this note is to obtain an effective error term on this asymptotic, under some additional hypotheses, and using more elementary techniques.
\begin{theorem}
Let $f_t$ be a family of endomorphisms of $\PP^N$ parametrized over $\PP^1_K$ for some number field $K\subseteq \CC$. Then
\[L(f_t)=L(f_\eta)\log|t^{-1}|+O\left(\log|t^{-1}|^{1-\epsilon_N}\right)\]
as $t\to 0$, with $\epsilon_N=1/(3N+4)$. 
\end{theorem}




\section{Effective estimates on norms of homogeneous forms}

Let $K\subseteq \CC$ be a number field, and let $\Phi\in K[\mathbf{X}, s]$ be homogeneous in the variables $\mathbf{X}=X_0, ..., X_N$, with coefficients in $K[s]$. We will write \[\|\Phi\|=\max\{|c|:c\text{ is a coefficient of }\Phi\},\]
(as a polynomial in $\mathbf{X}$ and $s$) and for a value $t\in\CC$ we will write $\Phi_t$ for the specialization at $s=t$, and set
$\|\Phi_t\|$ to be the analogous quantity for the specialized polynomial (as a homogeneous form in $\mathbf{X}$). We will write
\[\deg_s(\Phi)=\max\{\deg_s(c):c\text{ is a coefficient of }\Phi\}\]
(as a homogeneous form in $\mathbf{X}$ with coefficients in $K[s]$), and $\deg_{\mathbf{X}}(\Phi)$ for the degree of $\Phi$ as a homogeneous form.

Let $M_K$ be the usual set of places $v$ of $K$ indexing the set of absolute values $| \cdot |_v$ of $K$. We set $n_v = [K_v : \mathbb{Q}_v]/[K : \QQ]$, and normalize our absolute values so that
\begin{equation}\label{eq:prodfla} \sum_{v \in M_K}n_v \log |x|_v =0 \end{equation}
for $x \neq 0$.  The height of the projective tuple $[x_0:\cdots : x_N]$ is defined as
\begin{equation}\label{eq:heightdef} h([x_0:\cdots :x_N])=\sum_{v \in M_K}n_v \log \max_i|x_i|_v, \end{equation}
which is independent of the representative coordinates by~\eqref{eq:prodfla}.
For an algebraic number $\alpha$, we set $h(\alpha):=h([\alpha:1])$. As it will be useful below, we note that the height gives a trivial lower bound on any absolute value:  we have for any $v$ that
\[n_v \log |a|_v = -\sum_{w\neq v}n_w\log|a|_w \geq -\sum_{w\neq v}n_w\log^+|a|_w\geq -[K:\QQ]h(a),\] by~\eqref{eq:prodfla} and~\eqref{eq:heightdef},
whence
\begin{equation}\label{eq:liouville}|a|_v\geq e^{-[K:\QQ]h(a)/n_v}\geq e^{-[K:\QQ]h(a)}\end{equation}
as $n_v\geq 1$. 

 As above, we write \[\|\Phi\|_v=\max\{|c|_v:c\text{ is a coefficient of }\Phi\},\]
(as a polynomial in $\mathbf{X}$ and $s$) and for a value $t\in\CC$ we will write $\Phi_t$ for the specialization at $s=t$, and set
$\|\Phi_t\|_v$ to be the analogous quantity for the specialized polynomial (as a homogeneous form in $\mathbf{X}$)

For convenience of notation, we will work around $s=\infty$, and replace $s$ by $1/s$ later to obtain results around $s=0$. We also write $\deg^+=\max\{\deg, 1\}$ and $\log^+ =\max\{\log, 0\}$. Finally, we set
\[h^+(\Psi)=\sum_{v\in M_K}n_v\log^+\|\Psi\|_v.\]
Note that this is the height of the coefficients of $\Psi$ as a point in affine space (rather than the more traditional projective height), and so is \emph{not} invariant under scaling the coefficients.
\begin{lemma}\label{lem:spec}
Let $\Psi$ be a homogeneous form in the variables $\mathbf{X}=X_0, ..., X_N$, with coefficients in $K[s]$ for some number field $K$, suppose that the coefficients of $\Psi$ have no common root,  and write
\begin{equation}\label{eq:thetadef}\Theta(\Psi)=\max\{1,h^+(\Psi)+\deg_s(\Psi)+\deg_X(\Psi) \}.\end{equation} 
 Then
\[
\log\|\Psi_t\|=\deg_s(\Psi)\log^+|t|+O\Big(\deg_s^+(\Psi)\Theta(\Psi)\Big),\] where  the implied constant depends only on $N$.
\end{lemma}

\begin{proof}
We first note the useful inequality
\[\log^+\|\Psi\|\leq [K:\QQ]h^+(\Psi),\]
from the definition of $h^+$.
The statement is trivially true if $\deg_s(\Psi)=0$, since then $\|\Psi_t\|=\|\Psi\|$ is independent of $t$, and so $\log\|\Psi_t\|\leq [K:\QQ]\Theta(\Psi)$ by the non-negativity of the other terms in~\eqref{eq:thetadef}. Hence, we can suppose that $\deg_s(\Psi)\geq 1.$

Let $D=\deg_\mathbf{X}(\Phi)$, and let $\mathfrak{M}$ be the set of monomials of degree $D$ in $\mathbf{X}$. In one direction, the claimed statement is easy due to the triangle inequality. In fact, we may write
\[\Psi(\mathbf{X})=\sum_{\mathfrak{m}\in\mathfrak{M}}\psi_\mathfrak{m}\mathfrak{m},\]
where $\psi_{\mathfrak{m}}\in K[s]$. The specialization at $s=t$ is then just
\[\Psi_t(\mathbf{X})=\sum_{\mathfrak{m}\in\mathfrak{M}}\psi_\mathfrak{m}(t)\mathfrak{m}.\]
Since we have
\[|\psi_\mathfrak{m}(t)|\leq  \max\{1, |t|\}^{\deg_s(\psi_{\mathfrak{m}})}\cdot(\deg_s(\psi_\mathfrak{m})+1)\|\psi_\mathfrak{m}\|,\]
by the triangle inequality, we may take logs and obtain
\begin{align*}
    \log\|\Psi_t\| &=\max_{\mathfrak{m}\in\mathfrak{M}}\log|\psi_{\mathfrak{m}}(t)|\\
    &\leq \left(\max_{\mathfrak{m}\in\mathfrak{M}}\deg_s(\psi_{\mathfrak{m}})\right)\log^+|t| + \max_{\mathfrak{m}\in\mathfrak{M}}\log\|\psi_\mathfrak{m}\|+\max_{\mathfrak{m}\in\mathfrak{M}}\log(\deg_s(\psi_\mathfrak{m})+1)\\
    &=\deg_s(\Psi)\log^+|t|+\log\|\Psi\|+\log(1+\deg_s(\Psi)).
\end{align*}
Note that the error term here is better than what was claimed.

In other direction, since we are assuming that the coefficients $\psi_\mathfrak{m}$ of $\Psi(X)$ have no common factor, we can use \cite[Theorem 1.1]{J} to solve
\begin{gather}
at^{2\deg_s(\Psi)-1}=\sum_{\mathfrak{m}\in\mathfrak{M}} \psi_\mathfrak{m} g_\mathfrak{m}\\
a=\sum_{\mathfrak{m}\in\mathfrak{M}} \psi_\mathfrak{m} h_\mathfrak{m}
\end{gather}
for polynomials $g_\mathfrak{m}$, $h_\mathfrak{m}$ of degree at most $\deg_s(\Psi)-1$  in $K[s]$, and $a\neq 0$ in $K$. Suppose that we can find a solution with
\[\log\|g_\mathfrak{m}\|, \log\|h_\mathfrak{m}\|\leq C_1\] for all $\mathfrak{m}\in\mathfrak{M}$
and
\[-\log|a|\leq C_2,\]
where $C_1$ and $C_2$ will depend on some way on $\Psi$.

If $|t|>1$, we then estimate, using the upper bound $\#\mathfrak{M}\leq (D+1)^N$ and (6),
\begin{align*}
(2\deg_s(\Psi)-1)\log|t|+\log|a|& \leq \max_{\mathfrak{m}\in\mathfrak{M}}\log|\psi_\mathfrak{m}(t)|+ \max_{\mathfrak{m}\in\mathfrak{M}}\log|g_\mathfrak{m}(t)|\\&\qquad +N\log(\deg_X(\Psi)+1)\\
 &\leq \log\|\Psi_t\|+\max_{\mathfrak{m}\in\mathfrak{M}}\deg(g_\mathfrak{m})\log^+|t|+\max_{\mathfrak{m}\in\mathfrak{M}}\log\|g_\mathfrak{m}\|\\ 
 & \qquad+\max_{\mathfrak{m}\in\mathfrak{M}}\log(\deg(g_\mathfrak{m})+1) \\ &\qquad + N\log(\deg_X(\Psi)+1)\\
 &\leq \log\|\Psi_t\|+(\deg_s(\Psi)-1)\log^+|t|+C_1+\log\deg_s(\Psi)\\
 &\qquad+ N\log(\deg_X(\Psi)+1).\end{align*}
Thus, \[ \deg_s(\Psi)\log|t|\leq C_1+C_2+\log\|\Psi_t\|+\log(\deg_s(\Psi))+N\log(\deg_X(\Psi)+1)),\]


On the other hand, if $|t|<1$, we use (5) to get 

\begin{multline*}
    \log |a|\leq \log\|\Psi_t\|-\max_{\mathfrak{m}\in\mathfrak{M}}\deg(g_\mathfrak{m})\log|t|\\
     +\max_{\mathfrak{m}\in\mathfrak{M}}\log\|h_\mathfrak{m}\| +\max_{\mathfrak{m}\in\mathfrak{M}}\log(\deg(h_\mathfrak{m})+1)+ N\log(\deg_X(\Psi)+1).
\end{multline*}
This also yields 
\[\deg_s(\Psi)\log|t|\leq C_1+C_2+\log\|\Psi_t\|+\log(\deg_s(\Psi))+N\log(\deg_X(\Psi)+1)).\]

Therefore, we obtained
\[\deg_t(\Psi)\log^+|t|\leq C_1+C_2+\log\|\Psi_t\|+\log(\deg_s(\Psi)+1)+N\log(\deg_X(\Psi)+1)),\]
and all that remains is to describe suitable constants $C_1$ and $C_2$. 

The key observation is that (5) and (6) define a system of linear equations in the coefficients of the powers of $t$. If this system has rank $r$, then there is a solution in which $a\neq 0$ and the coefficients of $g_\mathfrak{m}$ and $h_\mathfrak{m}$ are all determinants of $r\times r$ submatrices of the coefficient matrix (see \cite[Lemma 4]{MW}). The entries are just the coefficients of $\Psi$, so we have
\begin{equation} \log|a|, \log\|g_\mathfrak{m}\|, \log\|h_\mathfrak{m}\|\leq r\log \|\Psi\|+r\log r\end{equation}
where $r\leq 4\deg_s(\Psi)$ just by counting equations. If we replace $\log\|\Psi\|$ with $\log^+\|\Psi\|$, we may safely replace $r$ with this upper bound.

For a lower bound on $\log|a|$ we have to be a little more careful. Note that $\Psi$ has coefficients in a number field $K$, so $a\in K$. We then have
\[-\log|a|=\log|a^{-1}|\leq \log^+|a^{-1}|\leq [K:\QQ]h(a^{-1})= [K:\QQ]h(a).\]
Working as in (3) for every $|.|_v$ and $v \in M_K$, we have that
$$
\log^+|a|_v \leq r\log^+ \|\Psi\|_v+r\log^+| r|_v$$ 
for every $v \in M_K$, and thus that (summing over all places and combining with~\eqref{eq:liouville})
\[-\log|a|\leq [K:\QQ](rh^+(\Psi)+r\log r).\]
 This shows that we may take $C_1=\deg_s(\Psi)(\log^+\|\Psi\|+\log\deg_s(\Psi))$ and $C_2=\deg_s(\Psi)(h^+(\Psi)+\log\deg_s(\Psi))$.
 
\end{proof}
We now treat the case in which the coefficients of $\Psi$ do have a nontrivial common factor.

\begin{lemma}\label{lem:theta}
Let $\alpha$ be a monic polynomial in $s$ which divides every coefficient of $\Psi$. Then $\Theta(\Psi/\alpha)\leq 2N\Theta(\Psi)$
\end{lemma}

\begin{proof}
Note that
\[\log\|\Psi/\alpha\|+\log\|\alpha\|\leq \log\|\Psi\|+(N\deg_X(\Psi)+\deg_s(\Psi))\log 2\]
along with $\log\|\alpha\|\geq 0$ implies
\[\log^+\|\Psi/\alpha\|\leq \log^+\|\Psi\|+(N\deg_X(\Psi)+\deg_s(\Psi))\log 2.\]
By the corresponding non-archimedean estimate, we also have
\[h^+(\Psi/\alpha)\leq h^+(\Psi)+(N\deg_X(\Psi)+\deg_s(\Psi))\log 2.\]
It follows that
\begin{align*}
    \Theta(\Psi/\alpha)&=\max\{1,h^+(\Psi/\alpha)+\deg_s(\Psi)+\deg_X(\Psi) \}\\
   & \leq \max\{1,h^+(\Psi)+(2\log 2+1)\deg_s(\Psi)+(2N\log 2+1)\deg_X(\Psi) \}\\
   &\leq 2N\Theta(\Psi),
\end{align*}
since $N\geq 2$.
\end{proof}

Combining these, we can obtain a version of Lemma~\ref{lem:spec} even in the case that the coefficients of $\Psi$ have common roots.
\begin{lemma}\label{lem:strongspec}
Let $\Psi$ be a homogeneous form in $X_0, ..., X_N$ with coefficients in $K[s]$, let $\alpha(s)$ be the (monic) greatest common divisor of the coefficients of $\Psi$. There exists a constant $C$ depending just on the set of roots of $\alpha$ such that for $|t|>C$,
\[\log\|\Psi_t\|=\deg_s(\Psi)\log^+|t|+O(\deg_s^+(\Psi)\Theta(\Psi)+\log\|\alpha\|),\]
where the implied constant depends just on $N$.
\end{lemma}

\begin{proof}
By the triangle inequality,
\begin{multline}\label{eq:alphathing}
    \left|\log\|\Psi_t\|-\deg_s(\Psi)\log^+|t|\right|\leq \left|\log\|\Psi_t/\alpha_t\|-\deg_s(\Psi/\alpha)\log^+|t|\right|\\+\left|\log|\alpha_t|-\deg(\alpha)\log^+|t|\right|
\end{multline}
Using Lemmas~\ref{lem:spec} and~\ref{lem:theta}, the first term in the upper bound satisfies
\[\left|\log\|\Psi_t/\alpha_t\|-\deg_s(\Psi/\alpha)\log^+|t|\right|=O(\deg^+_s(\Psi/\alpha)\Theta(\Psi/\alpha))=O(\deg_s^+(\Psi)\Theta(\Psi)),\]
where the implies constant still depends only on $N$. On the other hand, if $|t|>2|\gamma|$ for every root $\alpha_\gamma =0$, then the triangle inequality gives
\begin{align*}
    \left|\log|\alpha_t|-\deg(\alpha)\log^+|t|\right|&\leq \log\|\alpha\|+\deg(\alpha)\log 2\\
    &\leq \log\|\alpha\|+\deg_s^+(\Psi)\Theta(\Psi),
\end{align*}
simply because $\deg(\alpha)\leq \deg_s(\Psi)$ and $\Theta(\Psi)\geq 1$. Since both terms are non-negative, we may replace both bounds by large multiples to conclude the proof of the lemma.
\end{proof}

\section{Lifts and pushforwards of divisors}

Here we choose a lift $F$ of $f:\PP^N\to\PP^N$, that is, a choice of homogeneous forms representing $f$, which we can suppose that has degree $d\geq 2$. For a homogeneous form $\Phi$, we set
\[F_*\Phi(\mathbf{X})=\prod_{F(\mathbf{Y})=\mathbf{X}}\Phi(\mathbf{Y}).\]
\emph{A priori} we have that $F_*\Phi$ is an element of the field obtained from $\CC(\mathbf{X})$ by adjoining all roots of $F(\mathbf{Y})=\mathbf{X}$, but it is not hard to check that $F_*\Phi$ is in fact a homogeneous form in $\mathbf{X}$ of degree $d^N\deg(\Phi)$. (See page 4 in \cite{IB})

 \begin{lemma}
 \[G_F(\Phi):=\lim_{k\to\infty}\frac{\log\|F_*^k\Phi\|}{d^{k(N+1)}}=\int_{\PP^N}\log\|\Phi\|_G\mu_f,\]
 where $\mu_f$ is the measure of maximal entropy associated to $f$.
 \end{lemma}
\begin{proof} See~\cite[Proposition~14]{IB}
\end{proof}
First, we need a lemma that shows that any finite collection of polynomials can be ``sampled" such that the output is comparable to the norm of the polynomial.

\begin{lemma}\label{lem:samp}
Let $\CC_v$ be a complete, non-archimedean field, and fix $q_1, ..., q_m\in \CC_v[s]$. Then there exist infinitely many $t\in\CC_v$ with $|t|_v=1$ such that
\[\Big||q_i(t)|_v-\|q_i\|_v\Big|\leq \kappa_v,\]
where $\kappa_v=0$ for nonarchimedean places, and \[\kappa_v=(\max\{\deg(q_i)\}+1)\log\left(2\sum_{i=1}^m\deg(q_i)\right)\]  for archimedean places.
\end{lemma}

\begin{proof}
For non-archimedean places, if $|t|_v=1$, then we have $|q_i(t)|_v\leq \|q_i\|_v$ immediately.
Since the Gauss norm is multiplicative, it suffices (for the other direction) to treat the case where the $q_i$ are all linear. Now, if $q_i(s)=\alpha_i s -\beta_i$, with $|\alpha_i|\neq |\beta_i|$, we can take any $t$ with $|t|=1$. If $|\alpha_i|=|\beta_i|$, then there are infinitely many $t$ with $|t|=1$ and $|t-\beta_i/\alpha_i|=1$. Intersecting these set for $i=1, ..., m$ still gives us infinitely many choices with $|t|=1$.

Now consider archimedean places. In one direction, if $|t|_v=1$, then the triangle inequality gives
\[\log|q_i(t)|_v\leq \log\|q_i\|_v+\log(\deg(q_i)+1).\]

In the other direction, let $N=\sum_{i=1}^m\deg(q_i)$ and let $\epsilon=2/N$. Note that, given any $N$ points in $\CC_v=\CC$, disks of radius $\epsilon$ around these points cover at most a subset of the unit circle of arc length
\[N\frac{\pi}{\sqrt{2}}\epsilon= \frac{2\pi}{\sqrt{2}}<2\pi,\]
and so there exists a point with $|t|_v=1$ which is not within $\epsilon$ of any of these points.

Now, if $|t|_v=1$ and $|t-\gamma|\geq \epsilon$, we have three cases to consider. First, if $|\gamma|\leq 1$, then $\|s-\gamma\|=1$, so we have $|t-\gamma|_v\geq \epsilon \|s-\gamma\|$. On the other hand, if $1<|\gamma|_v\leq 2$, we have $\|s-\gamma\|_v\leq 2$, and so $|t-\gamma|_v\geq \frac{\epsilon}{2}\log\|s-\gamma\|$. Finally, if $|\gamma|_v>2$, we have from $|t|_v<1$ that
$|t-\gamma|_v\geq \frac{1}{2}|\gamma|_v=\frac{1}{2}\|s-\gamma\|_v$. Either way,
\[|t-\gamma|_v\geq \frac{\epsilon}{2}\|s-\gamma\|_v=\frac{1}{N}\|s-\gamma\|_v,\]
and this holds for all roots of $q_i$ (and for all $i$). Assuming $q_i$ is monic (without loss of generality, looking at the difference we're bounding), we have
\begin{align*}
    \log|q_i(t)|&=\sum_{j=1}^{\deg(q_i)}\log|t-\gamma_i|\\
    &\geq \sum_{i=1}^{\deg(q_i)} \left( \log\|s-\gamma\|-\log N\right)\\
    &\geq \|q_i\|-(\deg(q_i)+1)\log N- \deg(q_i)\log 2
    \end{align*}
by Gelfond's Lemma ( \cite[p. 22, p. 27]{BG}). Note that $N$ is the sum of degrees, so the error should take this into account.
\end{proof}

\begin{lemma} Let $\Phi$ be a homogeneous form in $\bf{X}$ with coefficients in $K(s)$. Then
\[\log\|(F_t)_*^k\Phi_t\|=\deg_s(F_*^k\Phi)\log^+|t|+O(d^{3k(N+1)})\]
as $k\to\infty$ and $t$ sufficiently large, independent of $k$. The implied constant now depends on $F$, $\Psi$, and $N$, but not $t$ or $k$.
\end{lemma}

\begin{proof}
For each $k$, let $\alpha_k$ be the (monic) greatest common divisor (in $K[s]$) of the coefficients of $F^k_*\Phi$. We will show that the roots of $\alpha_k$ are contained in a finite set that does not depend on $k$, and that
\begin{equation}\label{eq:alphak}\log\|\alpha_k\|\leq C_0d^{k(N+1)}\end{equation}
\begin{equation}\label{eq:degk}\deg_s^+(F^k_*\Phi)\leq C_1d^{k(N+1)}\end{equation}
and
\begin{equation}\label{eq:thetak}\Theta(F^k_*\Phi)\leq C_2d^{2k(N+1)},\end{equation}
which, in light of Lemma~\ref{lem:strongspec}, will prove the lemma. 

First, for each $\beta\in\CC$, let $\|\cdot\|_\beta$ be the corresponding absolute value on $\CC[s]$, defined by
\[\|w\|_\beta=e^{-\operatorname{ord}_{s=\beta}}(w),\]
with
\[\|w_0, ..., w_m\|_\beta=\max\{\|w_0\|_\beta, ..., \|w_m\|_\beta\}\]
as expected. We also write $\|w\|_\infty=e^{\deg(w)}$.

Now, by Lemma 12 in \cite{IB}, there is a finite set $S$ depending on $F$ but not $\Phi$ such that 
\[\log\|F_*\Phi\|_\beta=d^{N+1}\log\|\Phi\|_\beta\]
for all $\beta\not\in S$. Enlarging $S$ to include all common roots of the coefficients of $\Phi$, we then have
\[\log\|F^k_*\Phi\|_\beta=0\]
for all $k$ and all $\beta\not\in S$. In other words, for every $k$ and every $\beta\not\in S$, some coefficient of $F^k_*\Phi$ is non-vanishing at $\beta$. It follows that the roots of $\alpha_k$ lie in $S$ for all $k$, proving the first claim.

Next, note that we have
\[\alpha_k(s)=\prod_{\beta\in S}(s-\beta)^{-\log\|F_*^k\Phi\|_\beta}.\]
Thus, by Gelfond's Lemma,  and [Lemmas 7 and 9, Ing22] applied to $\|\cdot \|_\beta$, we have
\begin{align*}
0 &\leq \log\|\alpha_k\|\\
&\leq \sum_{\beta\in S}-\log\|F^k_*\Phi\|_\beta(\log\|s-\beta\|+\log 2)\\
& = \sum_{\beta\in S}-\log\|F^k_*\Phi\|_\beta(\log^+|\beta|+\log 2)\\
&\leq \sum_{\beta\in S}(-d^{k(N+1)}G_{F, \beta}(\Phi)+C_\beta)(\log^+|\beta|+\log 2)
\end{align*}
with $G_{F, \beta}(\Phi)=\lim_{k\to\infty}\frac{\log\|F_*^k\Phi\|_\beta}{d^{k(N+1)}}$, and $C_\beta$ some constant depending on $F$ and $\beta$. In particular, this proves~\eqref{eq:alphak}.

To prove~\eqref{eq:degk}, we again cite \cite{IB}. In particular, since the right-hand-side is unbounded, we may consider just $\deg_s(F^k_*\Phi)$ and compute
\begin{align*}
    \deg_s(F^k_*\Phi)&=\log\|F^k_*\Phi\|_\infty\\
    &\leq d^{k(N+1)}G_{F, \infty}(\Phi)+C_\infty.
\end{align*}

Finally, to prove~\eqref{eq:thetak}, we claim that
\begin{equation}\label{eq:arithmeticheight}\log\|F_*\Phi\|_v\leq d^{N+1}\log\|\Phi\|_v+C_v\end{equation}
Note that we cannot simply apply the results of \cite{IB}, though, since $\|\Phi\|$ is the largest modulus of a coefficient of $\Phi$ as a polynomial in both $\mathbf{X}$ and $s$. Assuming~\eqref{eq:arithmeticheight}, and assuming that $C_v=0$ for almost all $v$, we can then sum over all places to obtain
\[h^+(F_*\Phi)\leq d^{N+1}h^+(\Phi)+C,\]
which by summing a geometric series gives
\[h^+(F_*^k\Phi)\leq d^{k(N+1)}(h^+(\Phi)+C').\] 

Now fix a place (nonarchimedean for now), and by Lemma~8(6) of~\cite{IB} we have, for any $t\in \CC_v$,
\[\log\|(F_t)_*\Phi_t\|_v\leq d^{N+1}\log\|\Phi_t\|_v+\deg(\Phi_t)C_t,\]
where $C_t$ depends on $N$, $d$, and the coefficients of $F$. We have $\deg(\Phi_t)=\deg(\Phi)$ for all but finitely many $t$, so by specializing we can ensure that $|q(t)|_v=\|q\|_v$ for all of these coefficients, and the coefficients of $\Phi$ and $F_*\Phi$. This shows that we have the same error term on the generic fibre:
\begin{align*}
    \log\|F_*\Phi\|_v&=\log\|(F_t)_*\Phi_t\|_v\\
    &\leq d^{N+1}\log\|\Phi_t\|_v+\deg(\Phi_t)C_t\\
    &= d^{N+1}\log\|\Phi\|_v+\deg(\Phi)C.
\end{align*}
We want to show that the error is as maximal as it could be, independent of $t$.

Specifically, 
\[C_t=d^N\lambda_{\mathrm{Hom}^N_d} (f_t)-d^N\log\|F_t\|_v+d^Nc_3+d^{N+1}N\log^+|2|,\] with $c_3$ depending only on $N$ and $d$, and 
\[\lambda_{\mathrm{Hom}^N_d} (f_t)=-\log|\operatorname{Res}(F_t)|_v+(N+1)d^N\log\|F_t\|_v.\] In particular, if $T$ is the set of coefficients of $F$, coefficients of $\Phi$, coefficients of $\Phi(F)$, and $\text{Res}(F)$, then for all $t$ as in the sampling lemma (Lemma 7) above, $C_t$ will be the same, and will depend only on $d$, $N$, $\|\operatorname{Res}(F)\|_v$, and $\|F\|_v$. 


For the archimedean places, again let $T$ be the set of coefficients of $F$, $\Phi$, $F_*\Phi$, and $\operatorname{Res}(F)$, and suppose that all of these polynomials have degree at most $B$. Since
\[m(\Phi_t)- \frac{N}{2}\log(\deg(\Phi_t)+1)\leq \log\|\Phi_t\|\leq m(\Phi_t)+N\deg(\Phi_t)\log 2\]
For all but finitely many $t$, we have $\deg(\Phi_t)=\deg(\Phi)$ and same for $F_*\Phi$. We use denote by $C(S)$ a constant depending on a sequence of numbers $S$.
Note that $C_{t,v}=C(\|F_t\|_v, |\operatorname{Res}(F_t)|_v)$ is linear in the logarithms of the relevant quantities, 
so that we have
\begin{multline*}
    C(\|F_t\|_v, |\operatorname{Res}(F_t)|_v)\leq C(\|F\|_v, |\operatorname{Res}(F)|_v)+O((B+1)\log(\#T\cdot B))\\=O_F((B+1)\log(\#T\cdot B)).
\end{multline*}
Thus, by the sampling lemma (Lemma 7) there is a $t$ with
\begin{align*}
    \log\|F_*\Phi\|_v&\leq \log\|(F_t)_*\Phi_t\|_v+(B+1)\log(\#T\cdot B)\\
    &\leq m((F_t)_*\Phi_t) + N\deg(\Phi)\log 2 + (B+1)\log(\#T\cdot B)\\
    &\leq d^{N+1}m(\Phi_t) + \deg(\Phi)C(\|F_t\|_v, |\operatorname{Res}(F_t)|, d, N) + (B+1)\log(\#T\cdot B)\\
    &\leq d^{N+1}\log\|\Phi_t\|_v +\frac{N}{2}\log(\deg(F_*\Phi)+1)+ \deg(\Phi)C(\|F_t\|_v, |\operatorname{Res}(F_t)|_v, d, N)\\
    & + (B+1)\log(\#T\cdot B)\\
    &\leq d^{N+1}\log\|\Phi\|_v+(d^{N+1}+1)(B+1)\log(\#T\cdot B)+\frac{N}{2}\log(\deg(F_*\Phi)+1) \\
    & +\deg(\Phi)C(\|F_t\|_v, |\operatorname{Res}(F_t)|_v, d, N)\\
    &\leq d^{N+1}\log\|\Phi\|_v+(d^{N+1}+1)(B+1)\log(\#T\cdot B)+\frac{N}{2}\log(d^N\deg(\Phi)+1) \\
    & +O_F(\deg(\Phi)(B+1)\log(\#T\cdot B)).
\end{align*}
We now estimate
\begin{align*}
    B&=\max\{\log\|F\|_\infty, \log|\operatorname{Res}|_\infty, \log\|\Phi\|_\infty, \log\|F_*\Phi\|_\infty\}\\
    &=\max\{\log\|\Phi\|_\infty, \log\|F_*\Phi\|_\infty\}+O_F(1)\\
    &=O_F(\log\|\Phi\|_\infty+1)
\end{align*}
by previous estimates. Similarly, $\# T$ is one more than the number of coefficients of $F$, $\Phi$, and $F_*\Phi$, and so we have
\[\# T =O_F(\deg(\Phi)^N+1),\]
so we have
\[
    \log\|F_*\Phi\|_v\leq  d^{N+1}\log\|\Phi\|_v +O_F(\deg(\Phi)\log\|\Phi\|_\infty\log(\deg(\Phi)^N\log\|\Phi\|_\infty)).
\]
Since $\deg(F^k_*\Phi)=d^{kN}\deg(\Phi)$ and $\log\|F_*^k\Phi\|_\infty \leq d^{k(N+1)}(\log\|\Phi\|_\infty+O_F(1))$, we obtain
\[
    \log\|F_*^{k+1}\Phi\|_v\leq  d^{N+1}\log\|F^k_*\Phi\|_v +O_{F, \Phi}(d^{k(2N+1)}k).
\]
It follows by induction that
\[\log\|F_*^{k}\Phi\|_v=O_{F, \Phi}(d^{k(2N+1+\delta))})\]
This then gives $h^+(F^k_*(\Phi))=O(d^{k(2N+1+\delta))})$
for any $\delta>0$, or just $\Theta(F^k_*(\Phi))=O(d^{2k(N+1))})$.

To obtain~(\ref{eq:thetak}), we also notice that
\[\deg_{\mathbf{X}}(F_*^k\Phi)=d^{kN}\deg_{\mathbf{X}}(\Phi)\]
and
\begin{align*}
\deg_{s}(F_*^k\Phi)&=\log\|F_*^k\Phi\|_\infty\\
&=O(d^{k(N+1)}(\log\|\Phi\|_\infty+1)),
\end{align*}
both of which are negligible with respect to the claimed bounds for $\Theta(F_*^k\Phi)$.



\end{proof}

\section{Variation of the escape rate}

Using again the function field absolute value $\|\cdot \|_\infty$ defined by
\[\|\alpha\|_\infty = e^{-\mathrm{ord}_{t=\infty}(\alpha)},\]
we have that
\[\deg_t(\Phi)=-\mathrm{ord}_{t=\infty}(\Phi):=-\min\{\mathrm{ord}_{t=\infty}(\Phi_j)\}=\log\|\Phi\|_\infty\]
Recall that
\[G_F(\Phi)=\lim_{k\to\infty}\frac{\log\|F_*^k\Phi\|}{d^{k(N+1)}},\]
so that $G_F(F_*\Phi)=d^{N+1}G_F(\Phi)$, and $G_F(\Phi)=\log\|\Phi\|+O(\deg(\Phi))$. From the second property, we also get
\[-\mathrm{ord}_{t=\infty}(F_*^k\Phi)=G_F(\Phi)+O(\deg(F_*^k\Phi))=G_F(\Phi)+O(d^{kN})\]
for fixed $\Phi$.
It is also true that
\[\log\|F_*^k\Phi_t\|=G_{F_t}(\Phi_t)+O(d^{kN}\log^+|t|)\]
\begin{lemma}\label{lem:main}
We have
\[G_{F_t}(\Phi_t)=G_F(\Phi)\log^+|t|+O(\log^+|t|^{(3N+3)/(3N+4)}).\]
\end{lemma}

\begin{proof} 
We proceed as in Silverman and Call-Silverman (\cite{call1993canonical}). Let $M=3N+3$, choose $k=k(t)$ so that
\[d^{k(M+1)}\leq\log^+|t|< d^{(k+1)(M+1)},\]
and estimate
\begin{align*}
    \left|G_{F_t}(\Phi_t)-G_F(\Phi)\log^+|t|\right|&=d^{-k(N+1)}\left|G_{F_t}((F_t)_*^k\Phi_t)-G_F(F_*^k\Phi)\log^+|t|\right|\\
    &\leq d^{-k(N+1)}\left|G_{F_t}((F_t)_*^k\Phi_t)-\log\|(F_t)^k_*\Phi_t\|\right|\\
    &\quad + d^{-k(N+1)}\left|\log\|(F_t)^k_*\Phi_t\|+\mathrm{ord}_{t=\infty}(F_*^k\Phi)\log^+|t|\right|\\
    &\quad + d^{-k(N+1)}\left|-\mathrm{ord}_{t=\infty}(F_*^k\Phi)-G_F(F_*^k\Phi)\right|\log^+|t|.
\end{align*}
We bound each term separately. It is well known from previous calculations that $$G_{F_t}(H)= \log \|H\| + O((\deg H) \log \|F_t\|)=\log \|H\| +O((\deg H)\log^+ |t|). $$For $H=(F_t)^k_*\Phi_t,$ this implies that
$$
\left|G_{F_t}((F_t)_*^k\Phi_t)-\log\|(F_t)^k_*\Phi_t\|\right|=O(\deg((F_t)^k_*\Phi_t)\log^+|t|)=O(d^{kN}\log^+|t|).
$$
Thus, the first term will easily be $O(d^{-k}\log^+|t|)$. The last term is automatically $$d^{-k(N+1)}\cdot O(\deg_t(\Phi)(d^{kN}+d^{k(N+1)}\log^+|t|))=O(d^{-k}\log^+|t|)).$$ The middle term, based on  Lemma 8, is $O(d^{kM})$. Since $$d^{-k}< (\log^+|t|)^{\frac{-1}{M+1}} \text{ and } d^{kM} \leq (\log^+|t|)^{\frac{M}{M+1}},$$ we thus bound all the three referred terms with a desired error term of magnitude size at most $(\log^+|t|)^{1-1/(M+1)}$, where $M=3N+3$.

\end{proof}

\begin{proof}[Proof of the main result]
Replace $s$ by $s^{-1}$ so that we are considering $L(f_s)$ as $s\to\infty$.
It follows from Lemma~6 and \cite{BB} that for $J_s=\det(DF_s)$, we have
\[L(f_s)=G_{F_s}(J_s),\]
and so from Lemma~\ref{lem:main} we have
\begin{equation}\label{eq:asymp}L(f_t)=G_F(J_F)\log^+|t|+O((\log^+|t|)^{1-\epsilon_N}),\end{equation}
with $\epsilon_N=1/(3N+4)$, as $t\to\infty$. Note that one can deduce that $L(f)=G_F(J_F)$ on the generic fibre from the machinery in~\cite{BFJ}, but in any case
 Favre's asymptotic~\cite{F}
\begin{equation}\label{eq:favre}L(f_t)=L(f)\log^+|t|+o(\log^+|t|)\end{equation}
combined with~\eqref{eq:asymp} ensures this.


\end{proof}

\begin{remark}
Gauthier, Okuyama, and Vigny~\cite{GOV} have proven the analogue of~\eqref{eq:favre} in the case of non-archimedean fields. Although our focus is on the complex case, the arguments in this paper apply over non-archimedean fields as well.
\end{remark}



\begin{thebibliography}{99}
\bibitem{BB} 
Giovanni Bassanelli and Fran\c{c}ois Berteloot. Bifurcation currents in holomorphic dynamics on $\mathbb{P}^k$. \emph{J.\ Reine Angew.\ Math.} \textbf{608} (2007), pp.~201--235.

\bibitem{BFJ} S\'{e}bastien Boucksom, Charles Favre, and Mattias Jonsson. Solution to a non-Archimedean Monge-Amp\`{e}re equation. 
\emph{J.\ Amer.\ Math.\ Soc.} \textbf{28} (2015), no.~3, pp.~617--667. 
	
\bibitem{BG}  Enrico Bombieri and Walter Gubler. \emph{Heights in Diophantine geometry}. New Mathematical Monographs, 4. Cambridge University Press, Cambridge, 2006.
	
	
	
\bibitem{call1993canonical} Gregory S.~Call and Joseph~H.~Silverman. Canonical heights on varieties with morphisms. \emph{Compositio Math.} \textbf{89} no.~2 (1993), pp.~163--205.


\bibitem{F} Charles Favre.
Degeneration of endomorphisms of the complex projective space in the hybrid space.  
\emph{J.\ Inst.\ Math.\ Jussieu} \textbf{19} no.~4 (2020), pp.~1141--1183. 

\bibitem{GOV} Thomas Gauthier, Y\^{u}suke Okuyama, and Gabriel Vigny. Approximation of non-archimedean Lyapunov exponents and applications over global fields. \emph{Trans.\ Amer.\ Math.\ Soc.} \textbf{373} no.~12 (2020), pp.~8963–-9011. 



\bibitem{I1} Patrick Ingram.  Minimally critical endomorphisms of $\mathbb{P}^N$.   (2020), preprint \texttt{arXiv:2006.12869}

\bibitem{IB} Patrick Ingram. Explicit canonical heights for divisors relative to endomorphisms of $\PP^N$. (2022), preprint \texttt{arXiv:2207.07206}



\bibitem{J} Zbigniew Jelonek. On the effective Nullstellensatz. \emph{Invent.\ Math.} \textbf{162} no.~1 (2005), pp.~1--17. 


\bibitem{MW} D.~W.~Masser and G.~W\"{u}stholz.
Fields of large transcendence degree generated by values of elliptic functions. 
\emph{Invent.\ Math.} \textbf{72} no.~3 (1983), pp.~407--464. 

\end{thebibliography}
\end{document}